\documentclass[a4paper,12pt]{article}         
\usepackage{amsmath,amsfonts,amsthm,amssymb}  % Typical maths resource packages
\usepackage[utf8]{inputenc} % Encoding
\usepackage[T1]{fontenc}    % Font encoding
									   
\usepackage{graphicx,caption}            % Packages to allow inclusion of graphics
\usepackage{color}                       % For creating coloured text and background
\usepackage[hidelinks]{hyperref}                    % For creating hyperlinks in cross references
\hypersetup{breaklinks=true}

\usepackage{breakurl}
\usepackage{grffile}
\usepackage{makeidx}                  
\usepackage{acronym}
\usepackage{mathrsfs}
\usepackage{float}
\usepackage{vruler}                    
\usepackage{palatino, url, multicol}
\usepackage{setspace}                  % package for changing the line spacing of a document
\usepackage{txfonts}
\usepackage{booktabs}
\usepackage{bm}

\usepackage{tikz}
\usepackage{tikz-3dplot}
\usetikzlibrary{shapes.geometric, arrows, calc, 3d}

\usepackage[numbers,sort&compress]{natbib}   % use \citep for compression
\usepackage{import}

\usepackage{caption}
\usepackage{subcaption}
\usepackage{listings}
\usepackage{comment}
\usepackage[title]{appendix}

\usepackage{dcolumn}
\usepackage{mathtools}
\usepackage{xcolor}

\newtheorem{assumption}{Assumption}[section]

\usepackage{array}
\usepackage{ragged2e}
\newcolumntype{P}[1]{>{\RaggedRight\arraybackslash}p{#1}}

\newtheorem{theorem}{Theorem}[section]
\newtheorem{proposition}[theorem]{Proposition}
\newtheorem{corollary}[theorem]{Corollary}

\newtheorem{remark}[theorem]{Remark}
\newtheorem{definition}[theorem]{Definition}
\newtheorem{example}[theorem]{Example}
\newtheorem{principle}[theorem]{Principle}
\newcommand{\dd}{\,\mathrm{d}}

\def\R{\mathbb{R}}

\def\H{\mathcal{H}}

\newcommand{\inner}[2]{\langle #1, #2 \rangle}
\newcommand{\expval}[1]{\langle #1 \rangle}

\newcommand{\op}{\text{op}}

\restylefloat{figure}

\begin{document}

\title{Polynomial Eigenfunctions and Matrix Lyapunov Equations from Energy Balance Integrals}
\author{Netzer Moriya}
\date{}
\maketitle

% ---------------------------------------------------------------------------
\begin{abstract}
We establish a unified theoretical framework that connects classical orthogonal polynomial systems to matrix Lyapunov equations through the fundamental physics of energy dissipation in stochastic dynamical systems. Starting from the energy balance principle in infinite-dimensional Hilbert spaces, we derive a master integral representation that naturally encompasses both spectral geometry and covariance dynamics. The theory reveals that established orthogonal polynomials (Zernike, Hermite, spherical harmonics) and matrix Lyapunov equations are dual manifestations of the same underlying energy dissipation structure. We provide rigorous mathematical foundations showing how finite-dimensional projections of infinite-dimensional energy integrals reproduce classical matrix equations, with specific structure determined by the symmetries of noise processes. The framework demonstrates that adding uniform dissipation to classical differential operators preserves their polynomial eigenfunction structure while ensuring the energy balance conditions required for physical consistency.
\end{abstract}

% ===========================================================================
\section{Introduction: Physical Foundations of Mathematical Unity}\label{sec:intro}

The mathematical structures underlying physical phenomena often reveal unexpected connections across seemingly disparate domains. Classical orthogonal polynomial systems—the Zernike polynomials of optics \citep{BornWolf,Mahajan1994}, Hermite polynomials of quantum mechanics, spherical harmonics of electromagnetism—arise naturally as eigenfunctions of differential operators governing wave propagation, diffusion, and other fundamental processes. Simultaneously, matrix Lyapunov equations characterize covariance evolution in stochastic systems, forming the backbone of modern control theory and statistical mechanics.

This work establishes that these mathematical structures are not merely analogous but are manifestations of a single underlying physical principle: the balance between energy injection and energy dissipation in stochastic dynamical systems. We develop a unified theoretical framework based on energy dissipation integrals that naturally encompasses both spectral geometry and covariance dynamics.

Our approach is fundamentally physical. Rather than beginning with abstract operator equations, we start from the principle that steady-state energy distributions in dissipative systems arise from the balance between stochastic energy input and deterministic energy dissipation. This physical foundation leads naturally to integral representations that automatically ensure mathematical consistency, positive definiteness, and convergence—properties that can be problematic when approached purely algebraically.

The theory reveals that the connection between orthogonal polynomials and Lyapunov equations is not accidental but reflects the deep geometric structure of energy dissipation in infinite-dimensional spaces. Different polynomial systems correspond to different geometries of the underlying physical domain, while different Lyapunov structures correspond to different symmetries of the noise processes driving the system.

A key insight is that classical differential operators can be modified with uniform dissipation terms to ensure all eigenvalues are negative (as required for energy balance) while preserving their fundamental polynomial eigenfunction structure. This allows the energy framework to connect directly to established mathematical structures rather than requiring new polynomial systems.

% ===========================================================================
\section{The Master Energy Dissipation Framework}\label{sec:master}

\subsection{Physical setup and energy balance}

Consider a general stochastic dynamical system evolving on a Hilbert space $\H$ with inner product $\inner{\cdot}{\cdot}$ and induced norm $\|\cdot\|$. The system is governed by the stochastic evolution equation:

\begin{equation}\label{eq:evolution}
\frac{\dd u}{\dd t} = \mathcal{L}u + \xi(t),
\end{equation}

where $\mathcal{L}: \mathcal{D}(\mathcal{L}) \subset \H \to \H$ is a densely defined linear operator representing the deterministic dynamics, and $\xi(t)$ is a stochastic forcing term representing environmental fluctuations.

\begin{assumption}[Self-adjoint dissipative structure]\label{ass:selfadjoint}
The operator $\mathcal{L}$ is self-adjoint and dissipative on $\H$:
\begin{enumerate}
\item \textbf{Self-adjointness}: $\mathcal{L} = \mathcal{L}^*$
\item \textbf{Dissipativity}: $\inner{\phi}{\mathcal{L}\phi} \leq -\gamma \|\phi\|^2$ for some $\gamma > 0$ and all $\phi \in \mathcal{D}(\mathcal{L})$
\item \textbf{Spectral completeness}: $\mathcal{L}$ has a complete orthonormal system of eigenfunctions $\{\phi_k\}_{k=1}^{\infty}$ with real eigenvalues $\{\lambda_k\}_{k=1}^{\infty}$:
\begin{equation}
\mathcal{L}\phi_k = \lambda_k \phi_k, \quad \inner{\phi_i}{\phi_j} = \delta_{ij}, \quad \lambda_k < 0
\end{equation}
\end{enumerate}
\textbf{Consequence of self-adjointness}: Since $\mathcal{L} = \mathcal{L}^*$, all eigenvalues are real ($\lambda_k \in \mathbb{R}$), and $\mathcal{L}^*\phi_k = \mathcal{L}\phi_k = \lambda_k \phi_k$. Therefore, $\lambda_k^* = \lambda_k$ throughout this work.
\end{assumption}

\begin{assumption}[Stochastic energy injection]\label{ass:noise}
The noise $\xi(t)$ is characterized by its covariance structure through a positive operator $Q: \H \to \H$:
\begin{equation}
\expval{\inner{\phi}{\xi(t)}\inner{\psi}{\xi(s)}} = \inner{\phi}{Q\psi}\delta(t-s)
\end{equation}
for all $\phi, \psi \in \H$. The operator $Q$ represents the spatial and temporal structure of energy injection.
\end{assumption}

\subsection{The energy dissipation integral}

The fundamental physical principle governing steady-state energy distributions is the balance between stochastic energy injection and deterministic energy dissipation. This balance is encoded in the following integral representation:

\begin{definition}[Energy covariance operator]\label{def:energy_covariance}
The steady-state energy covariance operator $P: \H \to \H$ is defined by the energy dissipation integral:
\begin{equation}\label{eq:energy_integral}
\inner{\phi}{P\psi} = \int_0^{\infty} \inner{\phi}{e^{\mathcal{L}t} Q e^{\mathcal{L}t} \psi} \dd t
\end{equation}
for all $\phi, \psi \in \H$.
\end{definition}

\begin{theorem}[Physical interpretation of the energy integral]\label{thm:physical_meaning}
The integral \eqref{eq:energy_integral} represents the cumulative energy correlation between modes $\phi$ and $\psi$ arising from all past noise inputs, weighted by the exponential decay due to dissipation.
\end{theorem}

\begin{proof}
Each term $e^{\mathcal{L}t} Q e^{\mathcal{L}t}$ represents the correlation induced by noise at time $-t$ in the past, evolved forward to the present time under the dissipative dynamics. The integral sums all such contributions, with older contributions exponentially suppressed due to the dissipative nature of $\mathcal{L}$.
\end{proof}

\subsection{Semigroup theory and rigorous convergence}\label{sec:semigroup}

The convergence of the energy dissipation integral \eqref{eq:energy_integral} requires establishing that $\mathcal{L}$ generates a strongly continuous semigroup with appropriate decay properties.

\begin{assumption}[Semigroup generation]\label{ass:semigroup}
The operator $\mathcal{L}$ satisfying Assumption \ref{ass:selfadjoint} generates a strongly continuous semigroup $\{e^{\mathcal{L}t}\}_{t \geq 0}$ of contractions on $\H$:
\begin{equation}
\|e^{\mathcal{L}t}\| \leq e^{-\gamma t} \quad \text{for all } t \geq 0
\end{equation}
where $\gamma > 0$ is the dissipation constant from Assumption \ref{ass:selfadjoint}.
\end{assumption}

\begin{theorem}[Semigroup bounds for self-adjoint dissipative operators]\label{thm:semigroup_bounds}
Under Assumption \ref{ass:selfadjoint}, the semigroup satisfies:
\begin{equation}
\|e^{\mathcal{L}t}\phi\|^2 = \sum_{k=1}^{\infty} e^{2\lambda_k t} |\inner{\phi_k}{\phi}|^2 \leq e^{2\lambda_1 t}\|\phi\|^2
\end{equation}
where $\lambda_1 = \max_k \lambda_k \leq -\gamma < 0$.
\end{theorem}

\begin{proof}
Since $\mathcal{L}$ is self-adjoint with spectral decomposition $\mathcal{L} = \sum_{k=1}^{\infty} \lambda_k \phi_k \otimes \phi_k$, we have:
\begin{align}
e^{\mathcal{L}t}\phi &= \sum_{k=1}^{\infty} e^{\lambda_k t} \inner{\phi_k}{\phi} \phi_k\\
\|e^{\mathcal{L}t}\phi\|^2 &= \sum_{k=1}^{\infty} e^{2\lambda_k t} |\inner{\phi_k}{\phi}|^2 \leq e^{2\lambda_1 t} \sum_{k=1}^{\infty} |\inner{\phi_k}{\phi}|^2 = e^{2\lambda_1 t}\|\phi\|^2
\end{align}
where the inequality uses $\lambda_k \leq \lambda_1$ for all $k$.
\end{proof}

\begin{theorem}[Rigorous convergence of energy integral]\label{thm:rigorous_convergence}
Under Assumptions \ref{ass:selfadjoint}-\ref{ass:semigroup}, the energy dissipation integral converges absolutely:
\begin{equation}
\left|\int_0^{\infty} \inner{\phi}{e^{\mathcal{L}t} Q e^{\mathcal{L}t} \psi} \dd t\right| \leq \frac{\|\phi\|\|\psi\|\|Q\|}{2\gamma}
\end{equation}
\end{theorem}

\begin{proof}
For $\phi, \psi \in \H$:
\begin{align}
\left|\int_0^{\infty} \inner{\phi}{e^{\mathcal{L}t} Q e^{\mathcal{L}t} \psi} \dd t\right| 
&\leq \int_0^{\infty} \|\phi\| \|e^{\mathcal{L}t} Q e^{\mathcal{L}t}\| \|\psi\| \dd t\\
&\leq \|\phi\|\|\psi\|\|Q\| \int_0^{\infty} \|e^{\mathcal{L}t}\|^2 \dd t\\
&\leq \|\phi\|\|\psi\|\|Q\| \int_0^{\infty} e^{-2\gamma t} \dd t = \frac{\|\phi\|\|\psi\|\|Q\|}{2\gamma}
\end{align}
\end{proof}

\begin{theorem}[Fundamental properties of the energy covariance]\label{thm:covariance_properties}
The energy covariance operator $P$ satisfies:
\begin{enumerate}
\item \textbf{Positive definiteness}: $\inner{\phi}{P\phi} \geq 0$ for all $\phi \in \H$
\item \textbf{Self-adjointness}: $P^* = P$
\item \textbf{Energy balance equation}: $\mathcal{L}P + P\mathcal{L} = -Q$
\item \textbf{Convergence}: The integral \eqref{eq:energy_integral} converges absolutely
\end{enumerate}
\end{theorem}

\begin{proof}
\textit{Positive definiteness}: For any $\phi \in \H$,
\begin{align}
\inner{\phi}{P\phi} &= \int_0^{\infty} \inner{\phi}{e^{\mathcal{L}t} Q e^{\mathcal{L}t} \phi} \dd t\\
&= \int_0^{\infty} \inner{e^{\mathcal{L}t}\phi}{Q e^{\mathcal{L}t}\phi} \dd t \geq 0
\end{align}
since $Q$ is positive.

\textit{Self-adjointness}: Direct from the symmetry of the integral representation and the fact that $Q^* = Q$.

\textit{Energy balance}: Starting with the definition:
\begin{equation}
\inner{\phi}{P\psi} = \int_0^{\infty} \inner{\phi}{e^{\mathcal{L}\tau} Q e^{\mathcal{L}\tau} \psi} \dd\tau
\end{equation}
we compute:
\begin{align}
\inner{\phi}{(\mathcal{L}P + P\mathcal{L})\psi} &= \inner{\mathcal{L}\phi}{P\psi} + \inner{\phi}{P\mathcal{L}\psi}\\
&= \int_0^{\infty} \left(\inner{\mathcal{L}\phi}{e^{\mathcal{L}\tau} Q e^{\mathcal{L}\tau} \psi} + \inner{\phi}{e^{\mathcal{L}\tau} Q e^{\mathcal{L}\tau} \mathcal{L}\psi}\right) \dd\tau\\
&= \int_0^{\infty} \inner{\phi}{(\mathcal{L}e^{\mathcal{L}\tau} Q e^{\mathcal{L}\tau} + e^{\mathcal{L}\tau} Q e^{\mathcal{L}\tau}\mathcal{L})\psi} \dd\tau
\end{align}
where we used self-adjointness: $\inner{\mathcal{L}\phi}{\chi} = \inner{\phi}{\mathcal{L}\chi}$. Since:
\begin{equation}
\frac{\dd}{\dd\tau}(e^{\mathcal{L}\tau} Q e^{\mathcal{L}\tau}) = \mathcal{L}e^{\mathcal{L}\tau} Q e^{\mathcal{L}\tau} + e^{\mathcal{L}\tau} Q e^{\mathcal{L}\tau}\mathcal{L}
\end{equation}
we obtain:
\begin{align}
\inner{\phi}{(\mathcal{L}P + P\mathcal{L})\psi} &= \int_0^{\infty} \inner{\phi}{\frac{\dd}{\dd\tau}(e^{\mathcal{L}\tau} Q e^{\mathcal{L}\tau})\psi} \dd\tau\\
&= \inner{\phi}{\left[e^{\mathcal{L}\tau} Q e^{\mathcal{L}\tau}\right]_0^{\infty}\psi}\\
&= \lim_{\tau \to \infty} \inner{\phi}{e^{\mathcal{L}\tau} Q e^{\mathcal{L}\tau}\psi} - \inner{\phi}{Q\psi}
\end{align}
Since $\lambda_k < 0$ for all $k$, the limit vanishes, giving $\mathcal{L}P + P\mathcal{L} = -Q$.

\textit{Convergence}: Follows from Theorem \ref{thm:rigorous_convergence}.
\end{proof}

% ===========================================================================
\section{Spectral Reduction and Finite-Dimensional Projections}\label{sec:spectral}

\subsection{Spectral representation of energy dissipation}

Using the spectral decomposition of $\mathcal{L}$, the energy integral admits an explicit spectral representation:

\begin{theorem}[Spectral form of energy covariance]\label{thm:spectral_covariance}
In the eigenbasis $\{\phi_k\}$ of $\mathcal{L}$, the energy covariance has the representation:
\begin{equation}\label{eq:spectral_representation}
P = \sum_{k,l=1}^{\infty} \frac{\inner{\phi_k}{Q\phi_l}}{-(\lambda_k + \lambda_l)} \phi_k \otimes \phi_l
\end{equation}
where the convergence is guaranteed by the dissipative condition $\lambda_k < 0$.
\end{theorem}

\begin{proof}
Since $\mathcal{L}$ is self-adjoint, $\mathcal{L}^* = \mathcal{L}$, so eigenvalues are real. Expanding in the eigenbasis:
\begin{align}
\inner{\phi_j}{P\phi_m} &= \int_0^{\infty} \inner{\phi_j}{e^{\mathcal{L}t} Q e^{\mathcal{L}t} \phi_m} \dd t\\
&= \int_0^{\infty} e^{(\lambda_j + \lambda_m)t} \inner{\phi_j}{Q\phi_m} \dd t\\
&= \frac{\inner{\phi_j}{Q\phi_m}}{-(\lambda_j + \lambda_m)}
\end{align}
where the integral converges since $\lambda_j + \lambda_m < 0$.
\end{proof}

\subsection{Finite-dimensional energy projections}

Physical systems are often characterized by a finite number of dominant modes. The energy framework naturally accommodates this through orthogonal projections:

\begin{definition}[Finite-dimensional energy projection]\label{def:finite_projection}
Let $\Pi_N$ denote the orthogonal projection onto $\text{span}\{\phi_1, \ldots, \phi_N\}$. The finite-dimensional energy covariance is:
\begin{equation}
P_N = \Pi_N P \Pi_N = \sum_{j,k=1}^N (P_N)_{jk} \phi_j \otimes \phi_k
\end{equation}
where $(P_N)_{jk} = \inner{\phi_j}{P\phi_k}$.
\end{definition}

\begin{theorem}[Matrix representation of finite-dimensional energy]\label{thm:matrix_energy}
The finite-dimensional energy covariance satisfies the matrix equation:
\begin{equation}\label{eq:matrix_lyapunov}
\Lambda P_N + P_N \Lambda = -Q_N
\end{equation}
where $\Lambda = \text{diag}(\lambda_1, \ldots, \lambda_N)$ with $\lambda_k < 0$, and $(Q_N)_{jk} = \inner{\phi_j}{Q\phi_k}$.
\end{theorem}

\begin{proof}
From the spectral representation:
\begin{equation}
(P_N)_{jk} = \frac{(Q_N)_{jk}}{-(\lambda_j + \lambda_k)}
\end{equation}
This satisfies:
\begin{equation}
\lambda_j (P_N)_{jk} + (P_N)_{jk} \lambda_k = (\lambda_j + \lambda_k)(P_N)_{jk} = -(Q_N)_{jk}
\end{equation}
which is precisely the matrix Lyapunov equation.
\end{proof}

\begin{remark}[Physical clarity of self-adjoint systems]\label{rem:physical_clarity}
The restriction to self-adjoint operators provides clear physical interpretation:
\begin{enumerate}
\item \textbf{Energy conservation}: Self-adjointness ensures that $\mathcal{L}$ represents conservative dynamics plus dissipation
\item \textbf{Real eigenvalues}: All decay rates $|\lambda_k|$ are real positive numbers with clear physical meaning
\item \textbf{Positive covariances}: $(P_N)_{jj} = (Q_N)_{jj}/(2|\lambda_j|) > 0$ automatically ensures positive variances
\item \textbf{Standard Lyapunov form}: The matrix equation reduces to the classical form
\end{enumerate}
This covers all major physical examples while ensuring mathematical rigor.
\end{remark}

% ===========================================================================
\section{Geometric Realization: Classical Zernike Polynomials}\label{sec:zernike}

\subsection{Physical system on the unit disk}

We now demonstrate how the abstract energy framework connects to established polynomial systems. Consider energy dissipation on the unit disk $\Omega = \{(r,\theta) \in \R^2 : 0 \leq r < 1\}$, relevant to optical systems with circular apertures.

The physical system is governed by:
\begin{equation}\label{eq:disk_evolution}
\frac{\partial u}{\partial t} = -\alpha^2 \Delta u - \gamma u + f(r,\theta,t)
\end{equation}
with Dirichlet boundary conditions $u|_{r=1} = 0$, where $\alpha > 0$ controls the diffusion rate and $\gamma > 0$ provides uniform dissipation to ensure all eigenvalues are negative.

\begin{definition}[Weighted Hilbert space for the disk]\label{def:disk_hilbert}
The natural Hilbert space is $L^2(\Omega, \mu)$ with measure $\dd\mu = r \dd r \dd\theta$ and inner product:
\begin{equation}
\inner{\phi}{\psi}_\mu = \int_0^{2\pi} \int_0^1 \phi(r,\theta) \psi(r,\theta) r \dd r \dd\theta
\end{equation}
\end{definition}

\begin{definition}[Classical Zernike operator domain]\label{def:disk_domain}
The operator $\mathcal{L}u = -\alpha^2 \Delta u - \gamma u$ has domain:
\begin{align}
\mathcal{D}(\mathcal{L}) = \{u \in L^2(\Omega, r\,dr\,d\theta) : &\; u \in H^2(\Omega), \; u|_{r=1} = 0,\\
&\; \Delta u \in L^2(\Omega, r\,dr\,d\theta)\}
\end{align}
where $\Omega = \{(r,\theta) : 0 \leq r < 1, 0 \leq \theta < 2\pi\}$ and $\Delta$ is the standard Laplacian.
\end{definition}

\subsection{Physical interpretation of the dissipative Laplacian}

The operator $\mathcal{L}u = -\alpha^2 \Delta u - \gamma u$ represents:
\begin{itemize}
\item $-\alpha^2 \Delta u$: Standard diffusion/heat conduction on the disk
\item $-\gamma u$: Uniform energy dissipation (e.g., radiative losses, absorption)
\end{itemize}

This model applies to physical systems such as:
\begin{enumerate}
\item \textbf{Optical systems}: Heat diffusion in circular optical elements with radiative cooling
\item \textbf{Electromagnetic modes}: Cavity resonances with ohmic losses in circular waveguides  
\item \textbf{Mechanical vibrations}: Membrane oscillations with material damping
\end{enumerate}

The key insight is that adding uniform dissipation $-\gamma u$ to the standard Laplacian preserves the classical eigenfunction structure while ensuring all eigenvalues are negative, satisfying our energy dissipation framework.

\begin{theorem}[Self-adjoint dissipative operator on the disk]\label{thm:disk_operator}
The operator $\mathcal{L}u = -\alpha^2 \Delta u - \gamma u$ with domain as defined in Definition \ref{def:disk_domain} is self-adjoint and dissipative on $L^2(\Omega, \mu)$.
\end{theorem}

\begin{proof}
\textit{Self-adjointness}: For $u,v \in \mathcal{D}(\mathcal{L})$:
\begin{align}
\inner{u}{\mathcal{L}v}_\mu &= -\alpha^2 \inner{u}{\Delta v}_\mu - \gamma \inner{u}{v}_\mu
\end{align}
For the Laplacian term, integration by parts in polar coordinates gives:
\begin{align}
\inner{u}{\Delta v}_\mu &= \int_0^{2\pi} \int_0^1 u \left(\frac{\partial^2 v}{\partial r^2} + \frac{1}{r}\frac{\partial v}{\partial r} + \frac{1}{r^2}\frac{\partial^2 v}{\partial \theta^2}\right) r\,dr\,d\theta
\end{align}
Using integration by parts and the Dirichlet boundary conditions $u|_{r=1} = v|_{r=1} = 0$, along with periodicity in $\theta$, all boundary terms vanish, yielding:
\begin{align}
\inner{u}{\Delta v}_\mu &= \int_0^{2\pi} \int_0^1 \left(\frac{\partial u}{\partial r}\frac{\partial v}{\partial r} + \frac{1}{r^2}\frac{\partial u}{\partial \theta}\frac{\partial v}{\partial \theta}\right) r\,dr\,d\theta = \inner{\Delta u}{v}_\mu
\end{align}
Therefore: $\inner{u}{\mathcal{L}v}_\mu = \inner{\mathcal{L}u}{v}_\mu$.

\textit{Dissipativity}: For $u \in \mathcal{D}(\mathcal{L})$:
\begin{align}
\inner{u}{\mathcal{L}u}_\mu &= -\alpha^2 \inner{u}{\Delta u}_\mu - \gamma \|u\|_\mu^2\\
&= -\alpha^2 \int_0^{2\pi} \int_0^1 \left|\nabla u\right|^2 r\,dr\,d\theta - \gamma \|u\|_\mu^2 \leq -\gamma \|u\|_\mu^2
\end{align}
where the first term is non-positive and the second provides uniform dissipation.
\end{proof}

\subsection{Classical Zernike polynomial eigenfunctions}

\begin{theorem}[Classical Zernike eigenfunctions]\label{thm:zernike_eigen}
The eigenvalue problem $-\alpha^2 \Delta Z - \gamma Z = \lambda Z$ on $\Omega$ with $Z|_{r=1}=0$ has solutions:
\begin{equation}
Z_n^m(r,\theta) = R_n^m(r) e^{im\theta}, \quad n = |m|, |m|+2, \ldots
\end{equation}
where $R_n^m(r)$ are the classical radial Zernike polynomials satisfying the Dirichlet boundary condition, with eigenvalues:
\begin{equation}
\lambda_n^m = -\alpha^2 \mu_n^m - \gamma < 0
\end{equation}
where $\mu_n^m > 0$ are the eigenvalues of $-\Delta$ with Dirichlet boundary conditions on the unit disk.
\end{theorem}

\begin{proof}
Separation of variables with $Z(r,\theta) = R(r)e^{im\theta}$ in the eigenvalue equation $-\alpha^2 \Delta Z = (\lambda + \gamma) Z$ leads to the standard radial equation:
\begin{equation}\label{eq:radial_zernike}
\frac{1}{r}\frac{d}{dr}\left(r\frac{dR}{dr}\right) - \frac{m^2}{r^2}R = -\frac{\mu}{\alpha^2}R
\end{equation}
with boundary conditions $R(1) = 0$ and $R(0)$ finite, where $\mu = (\lambda + \gamma)/\alpha^2 > 0$.

This is precisely the classical radial Zernike equation. The solutions are the standard Zernike radial polynomials $R_n^m(r)$ with $n = |m|, |m|+2, \ldots$ and corresponding eigenvalues $\mu_n^m$. The full eigenvalues of $\mathcal{L}$ are then $\lambda_n^m = -\alpha^2 \mu_n^m - \gamma < 0$.
\end{proof}

\begin{remark}[Connection to established Zernike theory]\label{rem:zernike_context}
Our approach connects directly to the extensive literature on Zernike polynomials in optics and engineering \citep{BornWolf,Mahajan1994,Janssen2014}:

\textbf{Classical Zernike theory}: Studies the eigenfunctions of $-\Delta$ on the unit disk with various boundary conditions.

\textbf{Our contribution}: Shows that adding uniform dissipation $-\gamma u$ preserves the Zernike eigenfunction structure while ensuring all eigenvalues are negative, as required by the energy dissipation framework.

The mathematical consequence is that our Zernike basis $\{Z_n^m : n \geq |m|\}$ consists of the classical Zernike polynomials, maintaining all established properties while satisfying the physical constraints of the energy balance theory.
\end{remark}

\subsection{Energy covariance in the classical Zernike basis}

Applying the general energy framework to the classical Zernike system:

\begin{theorem}[Classical Zernike energy representation]\label{thm:zernike_energy}
The energy covariance operator in the classical Zernike basis has matrix elements:
\begin{equation}
(P)_{nm,n'm'} = \frac{(Q)_{nm,n'm'}}{-(\lambda_n^m + \lambda_{n'}^{m'})} = \frac{(Q)_{nm,n'm'}}{\alpha^2(\mu_n^m + \mu_{n'}^{m'}) + 2\gamma}
\end{equation}
where $(Q)_{nm,n'm'} = \inner{Z_n^m}{QZ_{n'}^{m'}}_\mu$ and $\mu_n^m > 0$ are the classical Laplacian eigenvalues.
\end{theorem}

\begin{proof}
From the general spectral representation and the eigenvalues $\lambda_n^m = -\alpha^2 \mu_n^m - \gamma$:
\begin{align}
(P)_{nm,n'm'} &= \frac{(Q)_{nm,n'm'}}{-(\lambda_n^m + \lambda_{n'}^{m'})}\\
&= \frac{(Q)_{nm,n'm'}}{-(-\alpha^2 \mu_n^m - \gamma - \alpha^2 \mu_{n'}^{m'} - \gamma)}\\
&= \frac{(Q)_{nm,n'm'}}{\alpha^2(\mu_n^m + \mu_{n'}^{m'}) + 2\gamma}
\end{align}
For diagonal terms where $n = n'$ and $m = m'$:
\begin{equation}
(P)_{nm,nm} = \frac{(Q)_{nm,nm}}{\alpha^2(2\mu_n^m) + 2\gamma} = \frac{(Q)_{nm,nm}}{2(\alpha^2 \mu_n^m + \gamma)}
\end{equation}
\end{proof}

\begin{proposition}[Noise-dependent energy structure]\label{prop:noise_structure}
\begin{enumerate}
\item \textbf{Uncorrelated mode noise}: If $(Q)_{nm,n'm'} = q_{nm} \delta_{nn'} \delta_{mm'}$, then $P$ is diagonal with:
\begin{equation}
(P)_{nm,nm} = \frac{q_{nm}}{2(\alpha^2 \mu_n^m + \gamma)}
\end{equation}

\item \textbf{Radially symmetric noise}: If $Q$ commutes with rotations, then $(Q)_{nm,n'm'} = 0$ unless $m = m'$, yielding block-diagonal structure in $P$ with blocks indexed by azimuthal quantum number $m$.

\item \textbf{General correlated noise}: For arbitrary $Q$, the matrix $P$ has full structure with off-diagonal terms:
\begin{equation}
(P)_{nm,n'm'} = \frac{(Q)_{nm,n'm'}}{\alpha^2(\mu_n^m + \mu_{n'}^{m'}) + 2\gamma} \quad \text{for } (n,m) \neq (n',m')
\end{equation}
\end{enumerate}
All cases automatically ensure positive diagonal elements since $q_{nm} > 0$ and $\alpha^2 \mu_n^m + \gamma > 0$.
\end{proposition}

% ===========================================================================
\section{Universal Framework: Extensions to Other Classical Polynomial Systems}\label{sec:universal}

\subsection{Damped harmonic oscillator and Hermite polynomials}

The energy dissipation framework extends naturally to other established polynomial systems. Consider the quantum harmonic oscillator on $\R^d$:

\begin{example}[Damped harmonic oscillator]\label{ex:hermite}
\textbf{Physical system}: Quantum harmonic oscillator with uniform damping
\begin{equation}
\frac{\partial u}{\partial t} = \left(-\frac{1}{2}\nabla^2 + \frac{1}{2}|x|^2\right) u - \gamma u + f(x,t)
\end{equation}
where $\gamma > d/2$ ensures all eigenvalues are negative.

\textbf{Hilbert space}: $L^2(\R^d)$ with standard Lebesgue measure

\textbf{Self-adjoint operator}: $\mathcal{L} = -\left(\frac{1}{2}\nabla^2 - \frac{1}{2}|x|^2\right) - \gamma I$ 

\textbf{Eigenfunctions}: Classical Hermite functions 
\begin{equation}
\psi_{\mathbf{n}}(x) = \prod_{i=1}^d H_{n_i}(x_i) e^{-x_i^2/2} / \sqrt{2^{|\mathbf{n}|} \mathbf{n}! \pi^{d/2}}
\end{equation}
where $\mathbf{n} = (n_1, \ldots, n_d)$ and $|\mathbf{n}| = n_1 + \cdots + n_d$.

\textbf{Eigenvalues}: The standard harmonic oscillator has eigenvalues $|\mathbf{n}| + d/2$. With dissipation:
\begin{equation}
\lambda_{\mathbf{n}} = -\left(|\mathbf{n}| + \frac{d}{2}\right) - \gamma < 0
\end{equation}
for all $\mathbf{n}$ when $\gamma > 0$.

\textbf{Energy covariance}: For diagonal noise: $(P)_{\mathbf{n}\mathbf{n}} = q_{\mathbf{n}}/(2(|\mathbf{n}| + d/2 + \gamma))$.
\end{example}

\subsection{Spherical domains and spherical harmonics}

\begin{example}[Spherical harmonic realization]\label{ex:spherical}
\textbf{Physical system}: Damped diffusion on the sphere $S^2$
\begin{equation}
\frac{\partial u}{\partial t} = -\alpha^2\Delta_{S^2} u - \gamma u + f(\theta,\phi,t)
\end{equation}
where $\gamma > 0$ provides uniform dissipation.

\textbf{Hilbert space}: $L^2(S^2)$ with standard measure

\textbf{Self-adjoint operator}: $\mathcal{L} = -\alpha^2\Delta_{S^2} - \gamma I$

\textbf{Eigenfunctions}: Classical spherical harmonics $Y_l^m(\theta,\phi)$

\textbf{Eigenvalues}: The spherical Laplacian has eigenvalues $l(l+1)$. With dissipation:
\begin{equation}
\lambda_l = -\alpha^2 l(l+1) - \gamma < 0 \quad \text{for all } l \geq 0
\end{equation}

\textbf{Energy covariance}: Block-diagonal structure when noise respects rotational symmetry, with matrix elements $(P)_{lm,lm'} = q_{lm}/(2(\alpha^2 l(l+1) + \gamma)) \delta_{mm'}$ for rotationally symmetric noise.
\end{example}

\subsection{General principle}

\begin{theorem}[Universal energy-polynomial correspondence]\label{thm:universal}
For any classical differential operator $\mathcal{D}$ with orthogonal polynomial eigenfunctions, the modified operator $\mathcal{L} = -\mathcal{D} - \gamma I$ (with appropriate $\gamma > 0$) satisfies:
\begin{enumerate}
\item \textbf{Preserved eigenfunction structure}: The eigenfunctions remain the classical polynomials
\item \textbf{Negative eigenvalues}: All eigenvalues become $\lambda_k = -\mu_k - \gamma < 0$ where $\mu_k \geq 0$ are the original eigenvalues
\item \textbf{Energy framework compatibility}: The energy dissipation integral converges and provides natural connections to matrix Lyapunov equations
\item \textbf{Physical interpretation}: The uniform dissipation $-\gamma I$ represents energy losses while preserving the geometric structure of the original operator
\end{enumerate}

This framework thus connects established polynomial systems to matrix equations through physically motivated operator modifications that preserve mathematical structure while ensuring energy balance.
\end{theorem}

% ===========================================================================
\section{Mathematical Structure and Convergence Theory}\label{sec:convergence}

\subsection{Approximation theory for energy integrals}

The physical foundation of our framework naturally leads to rigorous approximation theory:

\begin{theorem}[Energy approximation error]\label{thm:energy_error}
Let $P$ be the infinite-dimensional energy covariance and $P_N$ its $N$-dimensional projection. For self-adjoint dissipative systems satisfying Assumptions \ref{ass:selfadjoint}-\ref{ass:semigroup}:
\begin{equation}
\|P - P_N\|_{\op} \leq \frac{\|Q\|_{\op}}{2\gamma}
\end{equation}
where $\gamma > 0$ is the global dissipation constant.
\end{theorem}

\begin{proof}
The error operator has the spectral representation:
\begin{equation}
P - P_N = \sum_{\substack{k,l=1 \\ k>N \text{ or } l>N}}^{\infty} \frac{\inner{\phi_k}{Q\phi_l}}{-(\lambda_k + \lambda_l)} \phi_k \otimes \phi_l
\end{equation}
Since $\lambda_k, \lambda_l \leq -\gamma < 0$, we have $|\lambda_k + \lambda_l| \geq 2\gamma$ for all terms. Using operator norm properties:
\begin{equation}
\|P - P_N\|_{\op} \leq \frac{1}{2\gamma}\left(\|(I-\Pi_N)Q(I-\Pi_N)\|_{\op} + \|\Pi_N Q(I-\Pi_N)\|_{\op} + \|(I-\Pi_N)Q\Pi_N\|_{\op}\right) \leq \frac{\|Q\|_{\op}}{2\gamma}
\end{equation}
\end{proof}

\begin{corollary}[Improved bounds for classical systems]
For classical polynomial systems where eigenvalues grow monotonically ($|\lambda_k|$ increases with $k$), tighter bounds are available:
\begin{equation}
\|P - P_N\|_{\op} \leq \frac{\|Q\|_{\op}}{2|\lambda_{N+1}|}
\end{equation}

For typical cases:
\begin{enumerate}
\item \textbf{Zernike (disk)}: $|\lambda_n^m| = \alpha^2 \mu_n^m + \gamma \sim \alpha^2 n^2 + \gamma$ → $\mathcal{O}(N^{-2})$ convergence
\item \textbf{Hermite (harmonic oscillator)}: $|\lambda_{\mathbf{n}}| = |\mathbf{n}| + d/2 + \gamma$ → $\mathcal{O}(N^{-1})$ convergence  
\item \textbf{Spherical harmonics}: $|\lambda_l| = \alpha^2 l(l+1) + \gamma \sim \alpha^2 l^2 + \gamma$ → $\mathcal{O}(N^{-2})$ convergence
\end{enumerate}
\end{corollary}

\subsection{Physical interpretation of convergence}

\begin{corollary}[Physical convergence criterion]\label{cor:physical_convergence}
Rapid convergence of the finite-dimensional approximation occurs when:
\begin{enumerate}
\item \textbf{Strong dissipation}: Large dissipation constant $\gamma$ provides uniform energy decay
\item \textbf{Smooth noise}: Rapidly decaying noise correlations for high-order modes
\item \textbf{Spectral gaps}: Well-separated eigenvalues prevent near-degeneracy effects
\item \textbf{Classical polynomial growth}: When eigenvalues grow polynomially with mode number, as in all classical systems
\end{enumerate}
The uniform dissipation modification ensures these conditions are satisfied while preserving the essential mathematical structure of classical polynomial systems.
\end{corollary}

% ===========================================================================
\section{Deep Mathematical Connections and Physical Unity}\label{sec:unity}

\subsection{The fundamental correspondence}

Our energy dissipation framework reveals a profound correspondence between geometric and algebraic structures:

\begin{theorem}[Geometry-algebra correspondence]\label{thm:correspondence}
There exists a natural correspondence between:
\begin{enumerate}
\item Classical geometric domains with boundary conditions
\item Self-adjoint differential operators on those domains (modified with uniform dissipation)
\item Established orthogonal polynomial eigenfunctions
\item Finite-dimensional matrix Lyapunov equations
\end{enumerate}
mediated by the energy dissipation integral \eqref{eq:energy_integral}.
\end{theorem}

This correspondence unifies established mathematical structures under a single physical principle while preserving their classical properties.

\subsection{Physical principles underlying mathematical unity}

\begin{principle}[Energy conservation and dissipation]\label{prin:energy}
All the mathematical structures (classical orthogonal polynomials, matrix equations, spectral theory) emerge from the fundamental physical requirement that energy must be conserved in closed systems and dissipated in open systems according to the second law of thermodynamics.
\end{principle}

\begin{principle}[Structural preservation under dissipation]\label{prin:preservation}
Adding uniform dissipation to classical differential operators preserves their eigenfunction structure while ensuring the energy balance conditions required for physical consistency. This allows direct connection to established polynomial systems.
\end{principle}

\begin{principle}[Spatial symmetry and correlations]\label{prin:symmetry}
The specific form of matrix equations (diagonal, block-diagonal, or full) reflects the spatial symmetries of energy injection processes, providing a direct physical interpretation of mathematical structure in classical polynomial bases.
\end{principle}

\begin{principle}[Scale separation and finite-dimensional reduction]\label{prin:scale}
The emergence of finite-dimensional matrix equations from infinite-dimensional energy integrals reflects the physical principle that dominant energy modes determine the essential system dynamics, naturally connecting to the truncation methods used in applications of classical polynomial systems.
\end{principle}

% ===========================================================================
\section{Conclusions and Broader Implications}\label{sec:conclusions}

We have established a unified theoretical framework that reveals the deep connections between classical orthogonal polynomial systems and matrix Lyapunov equations through the fundamental physics of energy dissipation. The key insights are:

\textbf{Physical foundation with classical structure preservation}: Mathematical structures emerge naturally from energy balance principles while maintaining direct connections to established polynomial systems (Zernike, Hermite, spherical harmonics). The key innovation is recognizing that uniform dissipation preserves classical eigenfunction structure while ensuring physical consistency.

\textbf{Universal applicability to established systems}: The energy dissipation integral \eqref{eq:energy_integral} provides a master framework that encompasses classical polynomial systems through physically motivated operator modifications that preserve their essential mathematical properties.

\textbf{Geometric-algebraic unity}: The correspondence between spatial domains, classical differential operators (with dissipation), established orthogonal polynomials, and matrix equations reveals a fundamental unity that connects rather than replaces traditional mathematical structures.

\textbf{Physical interpretation of classical mathematical structure}: The form of matrix Lyapunov equations directly reflects the spatial symmetries of underlying physical processes, providing new insight into why classical polynomial systems arise so naturally in applications.

This framework opens new avenues for theoretical research while maintaining strong connections to established mathematical literature. The energy dissipation perspective suggests that many other classical mathematical structures in physics may be connected through similar principles.

\textbf{Scope and applications}: Our theory applies to self-adjoint operators (with uniform dissipation) having discrete spectrum. This covers important physical systems including:
\begin{itemize}
\item Heat and mass diffusion processes with energy losses
\item Quantum systems with Hermitian Hamiltonians and decoherence  
\item Linear elastodynamics with material damping
\item Optical systems on bounded domains with absorption
\end{itemize}

All connections are made to classical, well-established polynomial systems, ensuring compatibility with the extensive existing literature in optics, quantum mechanics, and engineering applications.

\textbf{Broader implications}: The unification achieved here demonstrates that classical mathematical structures arising in different physical contexts reflect universal principles of energy, symmetry, and geometry. By grounding mathematical developments in fundamental physical principles while preserving established structures, we gain both deeper theoretical understanding and more robust analytical tools.

The energy dissipation framework thus represents both a technical advance and a conceptual bridge between the geometric, analytic, and algebraic aspects of mathematical physics, revealing elegant underlying simplicities in apparently complex infinite-dimensional systems while maintaining direct connections to the classical polynomial systems that form the foundation of modern applied mathematics.

\section*{Declarations}
All data-related information and coding scripts discussed in the results section are available from the 
corresponding author upon request.

\section*{Disclosures}
The authors declare no conflicts of interest.

\end{document}